\documentclass[12pt]{amsart}
\usepackage{amsmath}
\usepackage{amsfonts}
\usepackage{amssymb}
\usepackage{graphicx}
\usepackage{aliascnt}
\usepackage[retainorgcmds]{IEEEtrantools}
\usepackage[bookmarks=true,pdfstartview=FitH, pdfborder={0 0 0}, colorlinks=true,citecolor=blue, linkcolor=blue]{hyperref}

\newtheorem{theorem}{Theorem}[section]
\newaliascnt{conj}{theorem}
\newaliascnt{cor}{theorem}
\newaliascnt{lemma}{theorem}
\newaliascnt{fact}{theorem}
\newaliascnt{claim}{theorem}
\newaliascnt{prop}{theorem}
\newaliascnt{definition}{theorem}

\newtheorem{lemma}[lemma]{Lemma}

\newtheorem{claim}[claim]{Claim}
\newtheorem{prop}[prop]{Proposition}
\newtheorem{definition}[definition]{Definition}

\aliascntresetthe{conj} \aliascntresetthe{cor}
\aliascntresetthe{lemma} \aliascntresetthe{fact}
\aliascntresetthe{claim} \aliascntresetthe{prop}
\aliascntresetthe{definition}

\theoremstyle{remark}
\newaliascnt{remark}{theorem}
\newtheorem{remark}[remark]{Remark}
\aliascntresetthe{remark}

\theoremstyle{remark}
\newaliascnt{exam}{theorem}
\newtheorem{exam}[exam]{Example}
\aliascntresetthe{exam} 

\numberwithin{equation}{subsection}
\newcommand{\vspan}{\operatorname{span}}
\newcommand{\diff}{\operatorname{Diff^+_0}}
\newcommand{\Hol}{\operatorname{Hol}}
\newcommand{\const}{\operatorname{const}}
\newcommand{\id}{\operatorname{Id}}

\newcommand{\p}{\partial}

\newcommand{\E}{\mathcal{E}}
\newcommand{\F}{\mathcal{F}}
\newcommand{\h}{\mathcal{H}}
\newcommand{\LL}{\mathcal{L}}
\newcommand{\U}{\mathcal{U}}

\newcommand{\RR}{\mathbb{R}}

\newcommand{\n}{\mathbf{n}}

\begin{document}

\title[On Legendrian foliations in contact manifolds]{On Legendrian foliations in contact manifolds I: Singularities and neighborhood theorems}

\author{Yang Huang}
\address{Max-Planck-Institut f\"ur Mathematik, Vivatsgasse 7, Bonn 53111\\Germany }
\email{yhuang@mpim-bonn.mpg.de}

\maketitle

\begin{abstract}
In this note we study several aspects of coisotropic submanifolds of a contact manifold. In particular we give a structure theorem for the singularity of the characteristic foliation of a coisotropic submanifold. Moreover we establish the existence and uniqueness results of germs of contact structures near Legendrian foliations, which is a special case of coisotropic submanifold. This note can be thought of as an attempt to generalize the study of surfaces in three-dimensional contact geometry to higher dimensions.
\end{abstract}

\section{introduction}
A contact manifold $(M^{2n+1},\xi)$ is a closed, oriented manifold with a hyperplane distribution $\xi=\ker\alpha$, where $\alpha$ is a 1-form on $M$ satisfying the condition that $\alpha\wedge(d\alpha)^n>0$. Conventionally, we say $M$ is a higher dimensional contact manifold if $n \geq 2$, as opposed to three dimensional contact geometry, which is much better understood currently. In this note we are primarily interested in higher dimensional contact geometry. So although all the results work equally well in dimension three, the conclusions are either trivial or well-known in that case.

For $n=1$ case, it is very important to understand embedded surfaces in $M$, and in particular germs of contact structures on them. There has been an extensive study on these topics in the past few decades. For example the characteristic foliation on surfaces was studied in \cite{Eli1989}, \cite{Eli1992}, and convex surface theory was studied in \cite{Gir1991}.

We wish to generalize our understanding of surfaces in dimension three to higher dimensions as much as possible, and this note can be considered as a small step in that direction. A reasonable analog of surfaces in higher dimension is coisotropic submanifolds, which will be defined in Section \ref{sec:singular_loci}. Briefly, a submanifold $Y \subset M$ of dimension greater than $n$ is {\em coisotropic} if the pointwise intersection $T_pY \cap \xi_p$ is a coisotropic subspace of $\xi_p$ with respect to the conformal symplectic form $d\alpha$, for any $p\in Y$. A coisotropic submanifold $(Y,\F)$ of dimension $n+1$ is called a {\em Legendrian foliation} because it naturally comes with a, possibly singular, foliation $\F$ with Legendrian leaves. To state the main results of this paper, we need one more definition. The {\em singular loci} $S(Y)$ of a coisotropic submanifold $Y \subset M$ is the set $\{p\in Y~|~T_p Y \subset \xi_p\}$. Some previously known constructions, including plastikstufe \cite{Nie2006}, \cite{Pre2007}, \cite{EP2011}, and bordered Legendrian open books (bLOB) \cite{PNW2013}, are all belong to the category of Legendrian foliations.

The goal of this note is to establish basic properties of coisotropic submanifolds. The main results are as follows.

\begin{theorem}\label{thm:sing_loci}
Let $Y^k \subset M^{2n+1}$ be a coisotropic submanifold of dimension $n+1 \leq k \leq 2n$. Then each path-connected component of $S(Y)$, if non-empty, is a submanifold without boundary of dimension $2n-k$, up to a $C^\infty$-small isotopy of $Y$. Moreover, when $k=n+1$, the normal bundle of each path-connected component of $S(Y)$, viewed as a 2-disk bundle, is flat\footnote{The flat structure in general does not respect the linear structure on the normal bundle. See Section \ref{subsec:nbhd_sing} for more details.}.
\end{theorem}

\begin{remark}
Note that $S(Y) \subset Y$ is a closed subset because it is the zero locus of a 1-form, but there may exist open path-connected components limiting on some closed components of $S(Y)$, which makes the structure more complicated than desired.
\end{remark}

\begin{remark}
The constructions of both plastikstufe and bLOB assume that the singular locus has symplectically trivial normal bundle. But in general a flat disk bundle is not necessarily trivial (if $n \geq 4$). Even the bundle is trivial, the flat connection may as well be nontrivial.
\end{remark}

\begin{remark}
\autoref{thm:sing_loci} also hold for immersed $Y$.
\end{remark}

To state the second result we need some preparations. Let $(Y,\F)$ be a $(n+1)$-dimensional manifold with a (singular) codimension one foliation $\F$, and let $S(Y)$ denote the singular locus of $\F$. Briefly speaking $S(Y)$ is a {\em normally controlled singularity} if the following holds:
\begin{enumerate}
	\item each component of $S(Y)$ is a closed $(n-1)$-submanifold,
	\item the normal bundle of each component of $S(Y)$ is flat,
	\item for each component of $S(Y)$, there exists a {\em covariant constant Liouville} (CCL) 1-form, adapted to $\F$, with respect to the flat connection. See Section \ref{subsec:nbhd_sing} and \ref{sec:nbhd_sing} for more details.
\end{enumerate}
Note however that a particular choice of a CCL 1-form is not part of the data for normally controlled singularity. We will always assume that $S(Y)$ is co-oriented, i.e., the normal bundle of $S(Y)$ is oriented. Now we state the existence and uniqueness results of germs of contact structures near a Legendrian foliation.

\begin{theorem}\label{thm:nbhd}
Let $(Y^{n+1},\F)$ be a (singularly) foliated manifold such that $\F$ has normally controlled singularity (might be empty). Then the following holds:
\begin{itemize}
	\item {\em(Existence)} There exists a rank $n$ real vector bundle $E$ over $Y$, and a contact structure $\xi$ defined in a neighborhood of the 0-section such that $(Y,\F)$ is a Legendrian foliation with respect to $\xi$.
	\item {\em(Uniqueness)} Suppose $Y \subset M^{2n+1}$ is a submanifold, and $\xi_0,\xi_1$ are two contact structures on $M$ such that $\F$ is a Legendrian foliation with respect to both contact structures. Then there exists a neighborhood $\U_i(Y)$ of $Y$, $i=0,1$, and a diffeomorphism $\phi:\U_0(Y) \to \U_1(Y)$ such that $\phi(Y)=Y$ and $\phi^\ast(\xi_1)=\xi_0$.
\end{itemize}
\end{theorem}

It might appear that the condition on the singularity of $\F$ is strange. But in fact, it is almost a necessary and sufficient condition for $\F$ to be a Legendrian foliation in some contact manifold. The only condition that is possibly removable is the exclusion of open components of $S(Y)$.

\begin{remark}
Singularities in both plastikstufe and bLOB trivially satisfy the conditions in \autoref{thm:nbhd}. In fact a neighborhood theorem for plastikstufe is known to K. Niederkr{\"u}ger \cite{Nie}.
\end{remark}

\begin{remark}
If $\F$ is a nonsingular foliation, then the condition on the singularity is void. Therefore \autoref{thm:nbhd} produces a unique (germ of) contact structure with $\F$ as the Legendrian foliation. Following a remarkable result of Thurston \cite{Th1976}, a nonsingular codimension one foliation $\F$ exists on $Y$ if and only if the Euler characteristic $\chi(Y)=0$. It is an interesting question to understand the relationship between the dynamics of $\F$ and the contact germ.
\end{remark}

This note is the first of a series of up-coming papers devoted to understanding the role of coisotropic submanifolds in higher dimensional contact geometry, including convex hypersurface theory and higher dimensional bypasses. In particular, the holomorphic theory of obstructing symplectic fillings will be discussed in a sequel paper \cite{HHprep}.

The organization of this note is as follows. In Section \ref{sec:singular_loci} we study the characteristic foliation and its singularities of a coisotropic submanifold, and in particular, Legendrian foliations. Then we establish neighborhood theorems for nonsingular and singular Legendrian foliations separately in Section \ref{sec:nbhd_thm}. Neighborhood theorem for general coisotrpic submanifolds is also possible but the initial datum are more complicated, so we will not discuss that part in this note.

\vspace{5mm}
\noindent
{\em Acknowledgements}. The author would like to thank Ko Honda for inviting him to visit Stanford University in 2013, where most material in this note was firstly presented. Thanks also go to Jian Ge and Thomas Vogel for many inspiring conversations during this work. Finally the author is grateful to the Max Planck Institute for Mathematics in Bonn for providing an excellent environment for research.

\vspace{5mm}

\section{The singular loci of Legendrian foliations}\label{sec:singular_loci}

\subsection{Coisotropic submanifolds and singular loci} \label{subsec:coiso}
Let $(M^{2n+1},\xi)$ be a $(2n+1)$-dimensional contact manifold, and choose a contact 1-form $\alpha$ such that $\xi=\ker\alpha$. Let $Y^k \subset M$ be a closed $k$-dimensional submanifold. The case of compact $Y$ with suitable boundary conditions will also be briefly discussed (\textit{cf.} \autoref{lem:coiso_boundary}). Define $Y_\xi$ to be the (singular) {\em characteristic distribution} of $\xi$ in $TY$, namely, $Y_\xi(p)=T_p Y \cap \xi_p$ for all $p \in Y$. A point $p \in Y$ is {\em singular} if $T_p Y \subset \xi_p$. Let $\lambda=\alpha|_Y \in \Omega^1(Y)$ be the restriction of $\alpha$ to $Y$. Define $$S(Y):=\{p \in Y~|~\lambda(p)=0\}$$ to be the set of singular points in $Y$.

\begin{definition}\label{defn:coisotropic}
A submanifold $Y \subset M$ is {\em coisotropic} if for any $p \in Y$, $Y_\xi(p) \subset \xi_p$ is a coisotropic subspace with respect to the symplectic form $d\alpha$, i.e., $(Y_\xi(p))^{\bot_{d\alpha}} \subset Y_\xi(p)$ where $\bot_{d\alpha}$ is the symplectic orthogonal complement.
\end{definition}

It is easy to see that the above definition is independent of the choice of $\alpha$. Moreover, note that by definition $\dim(Y) \ge n$ and $\dim(Y)=n$ if and only if $Y$ is Legendrian and $Y_\xi$ is the tangential distribution.

For the rest of this note, we will be mainly interested in the study of coisotropic submanifold $Y$ of $\dim(Y)=n+1$. In this case Frobenius integrability theorem implies that $Y_\xi$ can be integrated to a Legendrian foliation, i.e., a foliation with Legendrian leaves, away from $S(Y)$. In this case, we will call $Y_\xi$ the {\em characteristic foliation} on $Y$. However the notion of characteristic distribution and characteristic foliation do not coincide in dimensions greater than $n+1$. More details will be discussed in Section \ref{subsec:coiso_higher}.

The main goal of this section is to prove the following

\begin{prop}\label{prop:sing}
Each path-connected component of the singular locus $S(Y) \subset Y$ is $C^\infty$-generically a $(n-1)$-dimensional submanifold without boundary if nonempty.
\end{prop}

\begin{proof}
Suppose $S(Y)$ is nonempty, and therefore there exists $p \in Y$ such that $T_p Y \subset \xi_p$ is coisotropic. We first work locally in a Darboux chart. It follows from standard Moser's technique that there exists a neighborhood $U(p)$ of $p$ in $M$ contactomorphic to a neighborhood of the origin in $(\RR^{2n+1},\alpha)$, where $\alpha=dz-\sum_{i=1}^n y_idx_i$, such that $p$ is identified with $0\in\RR^{2n+1}$. Moreover we can assume that $U(p) \cap Y$ is identified with the graph $\Gamma(f)$ of a function
\begin{equation} \label{eqn:local_model}
f: \RR^{n+1}_{x_1,\cdots,x_n,y_n} \to \RR^n_{y_1,\cdots,y_{n-1},z}, \hspace{5mm}\text{where~} f(0)=0,~df(0)=0,
\end{equation}
in a neighborhood of the origin. It is convenient to write
\begin{gather*}
f(x_1,\cdots,x_n,y_n)=(y_1(x_1,\cdots,x_n,y_n),\cdots,y_{n-1}(x_1,\cdots,x_n,y_n),\\ z(x_1,\cdots,x_n,y_n))
\end{gather*}

Now we construct $n-1$ nonvanishing pointwise linearly independent vector fields in $\RR^{n+1}_{x_1,\cdots,x_n,y_n}$ as follows:
\begin{align*}
& \tilde V_1=\p_{x_1}-\frac{\p y_1}{\p y_n}\p_{x_n}+\frac{\p y_1}{\p x_n}\p_{y_n}, \\
& \cdots \\
& \tilde V_{n-1}=\p_{x_{n-1}}-\frac{\p y_{n-1}}{\p y_n}\p_{x_n}+\frac{\p y_{n-1}}{\p x_n}\p_{y_n},
\end{align*}
and consider their image in $\Gamma(f)$ given by
\begin{equation*}
\begin{split}
V_k=f_\ast(\tilde V_k) &=\p_{x_k}+\sum_{i=1}^{n-1}(\frac{\p y_i}{\p x_k}-\frac{\p y_k}{\p y_n}\frac{\p y_i}{\p x_n}+\frac{\p y_k}{\p x_n}\frac{\p y_i}{\p y_n})\p_{y_i}-\frac{\p y_k}{\p y_n}\p_{x_n} \\
& \quad +\frac{\p y_k}{\p x_n}\p_{y_n}+(\frac{\p z}{\p x_k}-\frac{\p z}{\p x_n}\frac{\p y_k}{\p y_n}+\frac{\p z}{\p y_n}\frac{\p y_k}{\p x_n})\p_z,
\end{split}
\end{equation*}
for $k=1,\cdots,n-1$. Therefore we have constructed $n-1$ nonvanishing pointwise linearly independent vector fields $\{V_1,\cdots,V_{n-1}\}$ on $\Gamma(f)$ in a neighborhood of the origin, and we claim the following is true.

\begin{claim} \label{claim:span_vf}
The vector fields $V_1,\cdots,V_{n-1}$ satisfies
\begin{enumerate}
    \item{$i_{V_k} \alpha=0, \hspace{2mm} \forall k \in \{1,\cdots,n-1\}$} \label{cond:isotropic}
    \item{$i_{V_k} d\lambda=0, \hspace{2mm} \forall k \in \{1,\cdots,n-1\},\text{ where } \lambda=\alpha|_{\Gamma(f)}\in\Omega^1(\Gamma(f))$;} \label{cond:preserve_sing}
    \item{$[V_k,V_l]=0, \hspace{2mm} \forall k,l \in \{1,\cdots,n-1\}$.} \label{cond:flat}
\end{enumerate}
\end{claim}

Using Cartan's formula, Claim~(\ref{cond:isotropic}) and~(\ref{cond:preserve_sing}) imply that
\begin{equation} \label{eqn:cov_const_form}
\mathcal{L}_{V_k} \lambda=0.
\end{equation}
In other words, the flow of $V_k$ preserves $\lambda$ for all $k$. The proof of the claim is a rather tedious but elementary calculation, so we postpone it to the end of the proof and continue to explain how the conclusions of the proposition (partially) follow from the claim. In fact Claim (\ref{cond:isotropic}) and~(\ref{cond:flat}) imply that $\vspan\{V_1,\cdots,V_{n-1}\}$ can be integrated to a $(n-1)$-dimensional isotropic foliation in a neighborhood of the origin in $\Gamma(f)$. Now~(\ref{eqn:cov_const_form}) and our assumption that $\lambda$ vanishes at $0$ imply that the leaf $S_0$ passing through $0$ is contained in $S(Y)$. Moreover, a tubular neighborhood of $S_0$, identified with a disk bundle on $S_0$, can be equipped with a flat connection such that the horizontal lifts of $S_0$ are exactly the leaves of the isotropic foliation, and it is such that $\lambda$ is ``covariant constant'' with respect to the flat connection. In fact such a construction determines the contact structure near the singular loci and will be discussed in more detail in Section~\ref{subsec:nbhd_sing}. For the moment the existence of such neighborhood guarantees that $S(Y)$ is  a submanifold of dimension at least $n-1$. Namely, the leaf passing through 0 gives a local chart diffeomorphic to $\RR^{n-1}$, and there can be neither self-intersections nor self-tangencies.

The contact condition asserts that $S(Y)$ cannot contain any open subsets of $Y$, so we only need to make sure a $n$-dimensional (not necessarily closed) submanifold $L \subset S(Y)$ can be perturbed away by a $C^\infty$-small isotopy of $Y$. Suppose such an $L$ exists, then it is Legendrian. Moreover the standard neighborhood theorem asserts that a tubular neighborhood $N(L)$ of $L$ in $M$ is contactomorphic to a neighborhood of $$\RR^n_{x_1,\cdots,x_n}=\{z=y_1= \cdots =y_n=0\}$$ in $(\RR^{2n+1},\alpha=dz-\sum_{i=1}^n y_idx_i)$, where $L$ is identified with $\RR^n_{x_1,\cdots,x_n}$ and $N(L) \cap Y$ is identified with $$\RR^{n+1}_{x_1,\cdots,x_n,y_1}=\{z=y_2=\cdots=y_n=0\},$$ such that the normal direction of $L$ in $Y$ is identified with $\p_{y_1}$. In this case, the restricted contact form $\lambda=\alpha|_{\RR^{n+1}_{x_1,\cdots,x_n,y_1}}=-y_1dx_1$, which vanishes exactly along $\{y_1=0\}$.

Consider a smooth function $g: \RR^{n+1}_{x_1,\cdots,x_n,y_1} \to \RR^n_{z,y_2,\cdots,y_n}$ defined by $$z=z(y_1), \hspace{2mm} y_2=\cdots=y_n=0.$$ We first show that the graph $\Gamma(g)$ of $g$ is also foliated by Legendrians. This is because $T\Gamma(g)=\vspan\{\p_{x_1},\cdots,\p_{x_n},\p_{y_1}+z'(y_1)\p_z\}$, and the only nontrivial thing to check is that
\begin{equation*}
\alpha \wedge d\alpha(\p_{x_k},\p_{x_1},\p_{y_1}+z'\p_z)=-y_k=0
\end{equation*}
for any $2 \leq k \leq n$.

Note that in this local model, the singular locus $S(\Gamma(g))=\{y_1=z'(y_1)=0\}$. Therefore we can choose a $C^\infty$-small function $z:\RR_{y_1} \to \RR$ such that it is compactly supported in a neighborhood of $0$ and $z'(0) \neq 0$. Then the Legendrian foliation becomes nonsingular. In the case when $L$ has nonempty boundary,  we can further require that $z|_{\p L}=0$ so that after perturbation only $\p L$ will be singular. A little care has to be taken in this case because $z$ will also depend on the coordinates on $L$ but it turns out that it does not affect our computation. By replacing $N(L) \cap Y$ with $\Gamma(g)$, we obtain the desired small isotopy which kills $L$ (or at least the interior of $L$) as a singular set.

Finally we wrap up the proof of \autoref{claim:span_vf}.
\begin{proof}[Proof of Claim~(\ref{cond:isotropic}) and~(\ref{cond:preserve_sing})]
We start by observing that
\begin{gather*}
T\Gamma(f)=\text{span}\{\p_{x_1}+\sum_{i=1}^{n-1}\frac{\p y_i}{\p x_1}\p_{y_i}+\frac{\p z}{\p x_1}\p_z,\cdots,\p_{x_n}+\sum_{i=1}^{n-1}\frac{\p y_i}{\p x_n}\p_{y_i}+\frac{\p z}{\p x_n}\p_z,\\ \p_{y_n}+\sum_{i=1}^{n-1}\frac{\p y_i}{\p y_n}\p_{y_i}+\frac{\p z}{\p y_n}\p_z\}.
\end{gather*}
By Frobenius integrability theorem and our assumption that $Y_\xi$ is a foliation, we have $$\alpha \wedge d\alpha|_{T\Gamma(f)}=dz \wedge (\sum_{i=1}^n dx_i \wedge dy_i)-\sum_{i=1}^n (y_i dx_i \wedge (\sum_{j \neq i} dx_j \wedge dy_j))|_{T\Gamma(f)}=0,$$ which is equivalent to the following set of equations:
\begin{align}
\label{eqn:foli1}
& (\frac{\p z}{\p x_a}-y_a)(\frac{\p y_b}{\p x_c}-\frac{\p y_c}{\p x_b}) - (\frac{\p z}{\p x_b}-y_b)(\frac{\p y_a}{\p x_c}-\frac{\p y_c}{\p x_a}) + (\frac{\p z}{\p x_c}- \\ & y_c)(\frac{\p y_a}{\p x_b}-\frac{\p y_b}{\p x_a})=0, \text{ for all } 1 \leq a,b,c \leq n-1; \nonumber \label{eqn:foli2} \\
& (\frac{\p z}{\p x_a}-y_a)\frac{\p y_b}{\p x_n} - (\frac{\p z}{\p x_b}-y_b)\frac{\p y_a}{\p x_n} + (\frac{\p z}{\p x_n}-y_n)(\frac{\p y_a}{\p x_b}-\frac{\p y_b}{\p x_a}) \\ & =0, \text{ for all } 1 \leq a,b \leq n-1; \nonumber \label{eqn:foli3} \\
& (\frac{\p z}{\p x_a}-y_a)\frac{\p y_b}{\p y_n} - (\frac{\p z}{\p x_b}-y_b)\frac{\p y_a}{\p y_n} + \frac{\p z}{\p y_n}(\frac{\p y_a}{\p x_b}-\frac{\p y_b}{\p x_a})=0, \\ & \text{ for all } 1 \leq a,b \leq n-1; \nonumber \label{eqn:foli4} \\
& \frac{\p z}{\p x_a} - (\frac{\p z}{\p x_n}-y_n)\frac{\p y_a}{\p y_n} + \frac{\p z}{\p y_n}\frac{\p y_a}{\p x_n} - y_a=0, \\ & \text{ for all } 1 \leq a \leq n-1, \nonumber
\end{align}
which are obtained by evaluating $\alpha \wedge d\alpha$ at all possible choices of three base vectors in $T\Gamma(f)$. Note that (\ref{eqn:foli1}) is redundant as it can be deduced from the other three equations but we include it here for completeness.

We first show that $\alpha(V_k)=0, \forall k\in\{1,\cdots,n-1\}$. This follows from a direct calculation as follows.
\begin{equation*}
\begin{split}
\alpha(V_k) &=(dz-\sum_{i=1}^n y_idx_i)(\p_{x_k}+\sum_{i=1}^{n-1}(\frac{\p y_i}{\p x_k}-\frac{\p y_k}{\p y_n}\frac{\p y_i}{\p x_n}+\frac{\p y_k}{\p x_n}\frac{\p y_i}{\p y_n})\p_{y_i} \\
& \quad -\frac{\p y_k}{\p y_n}\p_{x_n}+\frac{\p y_k}{\p x_n}\p_{y_n}+(\frac{\p z}{\p x_k}-\frac{\p z}{\p x_n}\frac{\p y_k}{\p y_n}+\frac{\p z}{\p y_n}\frac{\p y_k}{\p x_n})\p_z) \\
&=\frac{\p z}{\p x_k}-\frac{\p z}{\p x_n}\frac{\p y_k}{\p y_n}+\frac{\p z}{\p y_n}\frac{\p y_k}{\p x_n}-y_k+y_n\frac{\p y_k}{\p y_n} \\
&=0,
\end{split}
\end{equation*}
where the last equality follows from (\ref{eqn:foli4}).

Next we will show that $i_{V_k} d\alpha$ vanishes on $T\Gamma(f)$. To this end, we need the following identity
\begin{equation} \label{eqn:partial_y}
\Lambda_{a,b}:=\frac{\p y_a}{\p y_n}\frac{\p y_b}{\p x_n}-\frac{\p y_b}{\p y_n}\frac{\p y_a}{\p x_n}+\frac{\p y_a}{\p x_b}-\frac{\p y_b}{\p x_a}=0,
\end{equation}
for any $1 \leq a,b \leq n-1$. To prove (\ref{eqn:partial_y}), we plug (\ref{eqn:foli4}) into (\ref{eqn:foli2}) and (\ref{eqn:foli3}) and simplify to get
$$(\frac{\p z}{\p x_n}-y_n)(\frac{\p y_a}{\p y_n}\frac{\p y_b}{\p x_n}-\frac{\p y_b}{\p y_n}\frac{\p y_a}{\p x_n}+\frac{\p y_a}{\p x_b}-\frac{\p y_b}{\p x_a})=0$$
and
$$\frac{\p z}{\p y_n}(\frac{\p y_a}{\p y_n}\frac{\p y_b}{\p x_n}-\frac{\p y_b}{\p y_n}\frac{\p y_a}{\p x_n}+\frac{\p y_a}{\p x_b}-\frac{\p y_b}{\p x_a})=0$$
respectively. Arguing by contradiction, suppose (\ref{eqn:partial_y}) is not identically zero. Then by continuity there exists an open set $U$ on which $\Lambda \neq 0$. Therefore we have $\frac{\p z}{\p x_n}-y_n=\frac{\p z}{\p y_n}=0$ on $U$, but this is impossible because
$$1=\frac{\p y_n}{\p y_n}=\frac{\p^2 z}{\p y_n \p x_n}=\frac{\p}{\p x_n}(\frac{\p z}{\p y_n})=0,$$
therefore $\Lambda_{a,b}$ must be constantly equal to 0.

With this preparation, now we can show that
\begin{align*}
0 &= d\alpha(V_k,\p_{x_l}+\sum_{i=1}^{n-1}\frac{\p y_i}{\p x_l}\p_{y_i}+\frac{\p z}{\p x_l}\p_z) \\
  &= d\alpha(V_k,\p_{y_n}+\sum_{i=1}^{n-1}\frac{\p y_i}{\p y_n}\p_{y_i}+\frac{\p z}{\p y_n}\p_z),
\end{align*}
for all $1 \leq l \leq n$. Indeed, the first equality follows from
\begin{equation*}
d\alpha(V_k,\p_{x_l}+\sum_{i=1}^{n-1}\frac{\p y_i}{\p x_l}\p_{y_i}+\frac{\p z}{\p x_l}\p_z) =
    \begin{cases}
        \Lambda_{k,l}=0 & \text{if } 1 \leq l \leq n-1, \\
        \frac{\p y_k}{\p x_n}-\frac{\p y_k}{\p x_n}=0 & \text{if } l=n,
    \end{cases}
\end{equation*}
and the second equality follows from
$$d\alpha(V_k,\p_{y_n}+\sum_{i=1}^{n-1}\frac{\p y_i}{\p y_n}\p_{y_i}+\frac{\p z}{\p y_n}\p_z)=\frac{\p y_k}{\p y_n}-\frac{\p y_k}{\p y_n}=0.$$
This completes the proof of Claim~(\ref{cond:isotropic}) and (\ref{cond:preserve_sing}).
\end{proof}
\begin{proof}[Proof of Claim~(\ref{cond:flat})]
Since $[V_k,V_l]=[f_\ast(\tilde V_k),f_\ast(\tilde V_l)]=f_\ast[\tilde V_k,\tilde V_l]$, it suffices to show $[\tilde V_k,\tilde V_l]=0$ for any $0 \leq k,l \leq n-1$. This is done, again, by explicit calculation as follows.
\begin{equation*}
	\begin{split}
		[\tilde V_k,\tilde V_l] &=[\p_{x_k}-\frac{\p y_k}{\p y_n}\p_{x_n}+\frac{\p y_k}{\p x_n}\p_{y_n},\p_{x_l}-\frac{\p y_l}{\p y_n}\p_{x_n}+\frac{\p y_l}{\p x_n}\p_{y_n}] \\
					        &=\frac{\p \Lambda_{k,l}}{\p y_n} \p_{x_n}-\frac{\p \Lambda_{k,l}}{\p x_n}\p_{y_n}=0,
	\end{split}
\end{equation*}
where $\Lambda_{k,l}$ is as defined in (\ref{eqn:partial_y}).
\end{proof}
This completes the proof the proposition.
\end{proof}

\begin{remark}
The conclusions of \autoref{prop:sing} also hold for compact $Y$ with Legendrian boundary (\textit{cf.} \autoref{lem:coiso_boundary}).
\end{remark}

\begin{remark}
By definition $S(Y)$ is a closed subset of $Y$, but it is possible that some components of $S(Y)$ are open submanifolds. The limiting behavior of open components of $S(Y)$ onto a closed component is currently unclear.
\end{remark}

\subsection{Germs of contact structure near the singular loci} \label{subsec:nbhd_sing}

In this section we will take a closer look at the (germ of) contact structures in a tubular neighborhood of $S(Y)$ in $Y$. Often we will also write $S$ for $S(Y)$ when the ambient manifold is implicit. For simplicity we will assume in this section that $S$ is a connected closed $(n-1)$-dimensional submanifold of $Y$.

We start by reviewing some standard knowledge on connections on fiber bundles with the setup adapted to our purposes. Readers who are interested in more general theory of connections on fiber bundles are referred to \cite{Ehr1995}, \cite{Mor2001}.

Let $$D \to E \to B$$ be a disk bundle over $B$, where $B$ is a closed manifold and $D \subset \RR^2$ is the unit open disk. We will always abuse notations by identifying $B$ with $B \times \{0\} \subset E$ which is the $0$-section of $E$. We say $E$ is {\em vertically oriented} is each fiber is oriented.

\begin{remark}
The notion of being vertically oriented is equivalent to the 0-section $B \subset E$ being {\em co-oriented}. We will use the term ``co-oriented'' when we talk about submanifolds.
\end{remark}

\begin{remark}
The only reason we do not consider such $E$ as a rank 2 vector bundle is that the parallel transport, to be defined below, is not a linear map in general.
\end{remark}

An {\em Ehresmann connection}, or {\em connection} in brief, $A$ on $E$ is a pointwise decomposition $$T_pE=V_p \oplus H_p$$ which is smoothly varying with $p \in E$, and such that $V_p=T_p D$ for all $p \in E$, and $H_p=T_p B$ for all $p \in B$. In the literature, the subbundle $V$ is called the {\em vertical distribution} and the subbundle $H$ is called the {\em horizontal distribution}.  A connection $A$ is {called \em flat} if $H$ is integrable. A bundle $E$ is {\em flat} if it admits a flat connection. In this case, let $\h$ be the foliation on $E$ obtained by integrating $H$. In particular $B$ is a closed leaf of $\h$. It is easy to see that all the leaves of $\h$ are transverse to the fibers, and the projection from a leaf to $B$ is a covering map, at least near the 0-section. We use this observation to define a notion of parallel transport as follows.

\begin{definition}[Parallel transport]
For any path $\gamma:[0,1] \to B$, the {\em parallel transport} along $\gamma$ is a (partially defined) smooth map $$\Phi_\gamma: E_{\gamma(0)} \to E_{\gamma(1)}$$ such that $\Phi_\gamma(x)=\tilde\gamma(1) \in E_{\gamma(1)}$, where $x \in E_{\gamma(0)}$ and $\tilde\gamma:[0,1] \to E$ is the horizontal lift of $\gamma$ starting at $x$ using the covering map.
\end{definition}

Note that since our fiber is not closed, the above parallel transport may not be defined everywhere on $E$, but it is a diffeomorphism whenever it is defined. In fact, shrinking $D$ to a small disk around the origin if necessary, we can still define the holonomy representation $$\Hol: \pi_1(B) \to \diff(D)$$ using the parallel transport. Here $\diff D$ denotes the set of orientation-preserving self-diffeomorphisms of $D$ which fixes the origin.

Along the same lines, we define the covariant derivative associated with a flat connection $A$ as follows.

\begin{definition}[Covariant derivative]
Given any vector field $X$ on $B$, we define $\nabla_X: \Omega^k(E) \to \Omega^k(E)$ by $$\nabla_X \beta:=\LL_{\tilde X} \beta$$ for any $k$-form $\beta \in \Omega^k(E)$, where $\tilde X$ is the horizontal lift of $X$ on $E$.
\end{definition}

With the above preparations, now we consider the singular locus $S \subset Y$, which we assume to be a connected closed submanifold of codimension 2. A tubular neighborhood $N(S) \subset Y$ of $S$ can be identified with a disk bundle over $S$, which we will also denote by $N(S)$. The proof of \autoref{prop:sing} yields the following observation about the restricted contact form on $N(S)$.

\begin{lemma} \label{lem:flat_form}
Let $\lambda=\alpha|_{N(S)}$ be the restricted contact form as before. Then $N(S)$ is a vertically oriented flat disk bundle on $S$ such that $\lambda$ is covariant constant, i.e., $\nabla_X \lambda=0$ for any vector field $X$ on $S$.
\end{lemma}

\begin{proof}
It follows from \autoref{claim:span_vf} (\ref{cond:preserve_sing}) and (\ref{cond:flat}) that $\ker(d\lambda)$ defines a (codimension 2) foliation $\F$ on $N(S)$ such that $S$ is a closed leaf of $\F$. In particular $d\lambda$ is nondegenerate in the fiber direction and defines an orientation on it. Therefore by definition $N(S)$ is a vertically oriented flat disk bundle with the horizontal distribution given by $\F$. The fact that $\lambda$ is covariant constant with respect to this flat connection follows from (\ref{eqn:cov_const_form}).
\end{proof}

For later purposes, we also want to look at the converse to \autoref{lem:flat_form}, i.e., to construct (germs of) contact structures from a flat disk bundle. To this end, we introduce a so-called {\em covariant constant Liouville} (CCL) 1-form on $D \subset \RR^2$ as follows.

\begin{definition}\label{defn:good}
A 1-form $\beta$ on $D$ is {\em CCL} with respect to a flat connection on the disk bundle $D \to E \to B$ if the following holds:
    \begin{enumerate}
        \item $\beta$ is invariant under the action of $\Hol(\pi_1(B)) \subset \diff(D)$, \label{good_1}
        \item $\beta=0$ exactly at $0 \in D$, \label{good_2}
        \item $d\beta>0$ on $D$ with respect to the given orientation. \label{good_3}
    \end{enumerate}
\end{definition}

\begin{exam}
Suppose $E$ is a trivial flat bundle, namely, the holonomy $\Hol(\pi_1(B))=\text{id}$. Then $\beta=xdy-ydx$ on $\RR^2$ is a CCL 1-form. Following \autoref{lem:contact_germ} below, one can construct a contact structure on a bundle over $E$ such that $E$ is a Legendrian foliation with a ``parameterized elliptic singularity'', which shows up in the construction of both plastikstufe and bLOB.
\end{exam}

\begin{remark}
In general a CCL 1-form may not exist in a given flat disk bundle. In fact, even a nontrivial covariant constant 1-form does not always exist in general. It is an interesting question to find certain conditions on the holonomy group that guarantees the existence of CCL 1-forms.
\end{remark}

Now we state the converse to \autoref{lem:flat_form}. Let $\pi:\E \to E$ be a vector bundle on $E$ with fiber $\RR \times T^\ast \h$, where $\pi$ is the projection map and $\h$ is the horizontal foliation on $E$. We will denote by $p_\ast: TE \to T\h$ the horizontal projection map, thinking of $T\h$ as the horizontal distribution of a flat connection.

\begin{lemma}\label{lem:contact_germ}
Suppose $\beta$ is a CCL 1-form as defined above. Then there exists a contact structure $\xi$ on $\E$ such that $E$ is foliated by Legendrians with respect to $\xi$. Moreover $B$ is a co-oriented singular locus in $E$.
\end{lemma}

\begin{proof}
We first construct a 1-form $\eta$ on $\E$ which is similar to the canonical 1-form on cotangent bundle as follows. At each $(e,z,v) \in \E$, where $e \in E, z \in \RR, v \in T^\ast_e \h$, we define $\eta(w):=v(p_\ast(\pi_\ast w))$.

Next we construct a 1-form $\tilde \beta$ on $\E$ by parallel transporting $\beta$ to get a 1-form $\hat\beta$ on $E$, and then pull it back to $\E$ using the projection map $\pi$. Note that $\hat\beta$ is globally defined on $E$ by \autoref{defn:good} (\ref{good_1}), and moreover it is covariant constant by construction.

Finally we claim $\alpha:=dz+\tilde\beta-\eta$ is a contact form on $\E$ which satisfies all the desired conditions.  To check the contact condition, we work locally by choosing a (foliated) chart $$\phi: \RR^{n+1}_{x_1,\cdots,x_{n-1},s,t} \to U(p) \subset E$$ for $p \in B$, such that $\phi^\ast(\h)$ is the trivial foliation on $\RR^{n+1}$ with leaves $\{s=\const,t=\const\}$. Moreover $\sigma(s,t):=\phi^\ast(\hat\beta)$ is independent of the $x_1,\cdots,x_{n-1}$ variables, and $d\sigma$ is nondegenerate in the $(s,t)$-plane. Now $\phi$ induces a chart $$\psi: \RR^{2n+1}_{x_1,\cdots,x_{n-1},s,t,y_1,\cdots,y_{n-1},z} \to \U(p) \in \E$$ where $y_1,\cdots,y_{n-1}$ are the dual coordinates to $x_1,\cdots,x_{n-1}$ on $T^\ast \h$, and $z$ is the coordinate on $\RR$ as before. It is easy to see that $$\psi^\ast(\alpha)=dz+\sigma-\sum_{i=1}^{n-1} y_idx_i$$ is contact in this local chart. Therefore $\alpha$ is a contact form on $\E$.

The assertion that $E$ is foliated by Legendrians and $B$ is a singular locus is now obvious by construction because $\alpha|_E=\hat\beta$.
\end{proof}

As an application of our understanding of the restricted contact form near the singular loci, we briefly discuss here the case of compact Legendrian foliated submanifold $Y \subset M$ with Legendrian boundary.
\begin{lemma} \label{lem:coiso_boundary}
Let $Y \subset M$ be a compact coisotropic submanifold of $M$ with Legendrian boundary. Suppose $S(Y) \subset Y$ is a $(n-1)$-dimensional submanifold. Then each connected component of $S(Y)$ is either contained in $\p Y$ or contained in the interior of $Y$.
\end{lemma}

\begin{proof}
A connected component of $S(Y)$ cannot be tangent to, but not contained in, $\p Y$ due to the existence of a horizontal isotropic foliation in a tubular neighborhood. A connected component of $S(Y)$ also cannot transversely intersect $\p Y$ because if $p \in S(Y) \cap \p Y$ is a transversal intersection point, then we know from the proof of \autoref{lem:flat_form} that $d\lambda$ is nondegenerate on $T_pS(Y)^\bot$, where $\lambda$ is the restricted contact form. On the other hand, $d\lambda(p)=0$ because $T_pS(Y)^\bot \subset T_p (\p Y)$ is Legendrian. Therefore we have a contradiction.
\end{proof}

\subsection{Coisotropic submanifolds of arbitrary dimension} \label{subsec:coiso_higher}

Now we consider $k$-dimensional coisotropic submanifold $Y^k \subset M^{2n+1}$ with any $n+1 \leq k \leq 2n$.  In this case, we define the (singular) {\em characteristic foliation} $\F_\xi$ on $Y$ by $$\F_\xi = \ker (\alpha \wedge (d\alpha)^{k-n-1}),$$ where the singular set $S(Y)=\{p\in Y~|~T_p Y \subset \xi_p\}$. When $k=n+1$, our definition coincides with the definition of Legendrian foliation given at the beginning of Section \ref{subsec:coiso}. When $k=2n$, the coisotropy condition is void and $\F_\xi$ is a vector field. The well-definedness of $\F_\xi$ is the content of the following lemma.

\begin{lemma} \label{lem:char_fol}
The distribution $\F_\xi$ can be integrated to a $(2n-k+1)$-dimensional foliation away from $S(Y)$.
\end{lemma}

\begin{proof}
Away from $S(Y)$, it is easy to see that $\alpha \wedge (d\alpha)^{k-n-1}$ is a nonvanishing form of constant rank, so $\ker(\alpha \wedge (d\alpha)^{k-n-1})$ defines a $(2n-k+1)$-dimensional distribution on $Y$. The Frobenius integrability theorem asserts that it is a foliation if $$d(\alpha \wedge (d\alpha)^{k-n-1})=(d\alpha)^{k-n}=0$$ restricted to $Y$. This is obviously true because the rank of $d\alpha$, restricted to $\xi \cap TY$, is equal to $k-n-1$.
\end{proof}

We now give a statement on the dimension of $S(Y)$ for any coisotropic $Y$, similar to \autoref{prop:sing}. Again, we assume that $Y$ is closed.

\begin{prop}\label{prop:sing_high}
For a $C^\infty$-generic coisotropic submanifold $Y$ of dimension $k, n+1 \leq k \leq 2n$, each path-connected component of $S(Y)$ is a $(2n-k)$-dimensional submanifold without boundary if nonempty.
\end{prop}

\begin{proof}
The proof is essentially the same as the proof of \autoref{prop:sing}, so we only give a sketch here. Now $Y$ in a neighborhood of a singular point $p \in S(Y)$ is modeled on the graph of a function $$f: \RR^{k}_{x_1,\cdots,x_n,y_{2n-k+1},\cdots,y_n} \to \RR^{2n-k+1}_{y_1,\cdots,y_{2n-k},z}, \hspace{3mm}\text{such~that~} f(0)=0,~df(0)=0,$$ in the standard $(\RR^{2n+1},dz-\sum_{i=1}^n y_idx_i)$. This is analogous to (\ref{eqn:local_model}). All the calculations thereafter carries over to this case with little modification, so we leave the details to the interested reader. It is slightly trickier to perturb away higher dimensional singular loci, but it turns out the graphical perturbation used in the proof of \autoref{prop:sing} still works in this case.
\end{proof}

Now the conclusions of \autoref{thm:sing_loci} are just the combination of the conclusions of \autoref{prop:sing}, \autoref{prop:sing_high} and \autoref{lem:flat_form}.

\section{Neighborhood theorems for Legendrian foliations} \label{sec:nbhd_thm}

\subsection{The case of nonsingular Legendrian foliations}

We will show the existence and uniqueness of the germ of a contact structure near a closed (sub-)manifold with fixed nonsingular (Legendrian) foliation. To formulate the existence part more precisely, let $(Y^{n+1},\F)$ be a closed foliated manifold of dimension $n+1$, where $\F$ is a co-oriented codimension one foliation. Consider a vector bundle $E \to Y$ with fiber $T^\ast \F$, then we have the following

\begin{lemma}[Existence] \label{lem:exist_nonsing}
There is a contact structure $\xi$ on the total space of $E$ such that $\F$ is the characteristic foliation on $Y$.
\end{lemma}

\begin{proof}
The technique for constructing the contact structure is similar to the one used in the proof of \autoref{lem:contact_germ}. But here we need one more piece of data, namely, a line field $L$ on $Y$ transverse to $\F$. Then we have $TY=T\F \oplus L$. Let $p_1: TY \to T\F$ be the projection to the first factor. Let $\pi: E \to Y$ be the projection map. We define a ``tautological 1-form'' $\eta$ on $E$ as follows. For any $v \in T_{(p,w)}E$, where $p \in Y, w \in T^\ast_p \F$, define $\eta(v):=w(p_1(\pi_\ast v))$.

Now let $\beta\neq 0 \in \Omega^1(Y)$ be a defining 1-form of $\F$, i.e., $\ker \beta=\F$. Abusing notations, we will also write $\beta$ for $\pi^\ast(\beta) \in \Omega^1(E)$. Then we claim that $\alpha=\beta-\eta$ is a contact form on $E$. To see this, we compute locally by choosing a (foliated) chart 
\begin{equation}
	\phi: \RR^{n+1}_{t,x_1,\cdots,x_n} \to \U(p) \in Y
\end{equation}
near some point $p \in Y$, such that the leaves of $\F$ are given by $\{t=\text{const}\}$. Let $y_i$ be the dual coordinates to $x_i$ on $T^\ast\F$. Suppose that locally $L$ is spanned by the vector field
\begin{equation*}
	L=\langle \p_t+\sum_{i=1}^n R_i\p_{x_i} \rangle
\end{equation*}
where $R_i \in C^\infty(\U(p))$ for $1 \leq i \leq n$.

Locally we may write $\beta=fdt$ for some $f>0 \in C^\infty(\U(p))$. One easily sees that in these coordinates
\begin{equation*}
	\eta=-(\sum_{i=1}^n R_iy_i)dt+\sum_{i=1}^n y_idx_i.
\end{equation*}
So we have
\begin{equation*}
	\alpha=(f+\sum_{i=1}^n R_iy_i)dt-\sum_{i=1}^n y_idx_i.
\end{equation*}

Verifying $\alpha$ is indeed a contact form is a straightforward calculation which we will carry out explicitly as follows.
\begin{IEEEeqnarray*} {rCl}
	\alpha\wedge(d\alpha)^n &=& \Big( (f+\sum_{i=1}^n R_iy_i)dt-\sum_{i=1}^n y_idx_i \Big) \wedge \Big( \sum_{j=1}^n (\frac{\p f}{\p x_j}+\sum_{i=1}^n \frac{\p R_i}{\p x_j}y_i)dx_j \wedge dt  \\
			&& +\: \sum_{i=1}^n R_idy_i \wedge dt+\sum_{i=1}^n dx_i \wedge dy_i \Big)^n \\
			&=& n! \Big( (f+\sum_{i=1}^n R_iy_i)dt-\sum_{i=1}^n y_idx_i \Big) \wedge \Big( dx_1 \wedge dy_1 \wedge\cdots\wedge dx_n \wedge dy_n+\sum_{j=1}^n (\frac{\p f}{\p x_j} \\
			&& +\: \sum_{i=1}^n \frac{\p R_i}{\p x_j}y_i) dx_1 \wedge dy_1 \wedge\cdots\wedge dx_j\wedge\widehat{dy_j} \wedge\cdots\wedge dx_n \wedge dy_n \wedge dt \\
			&&+\: \sum_{i=1}^n R_i dx_1 \wedge dy_1 \wedge\cdots\wedge \widehat{dx_i}\wedge dy_i \wedge\cdots\wedge dx_n \wedge dy_n \wedge dt \Big) \\
			&=& n! ( f+\sum_{i=1}^n R_iy_i-\sum_{i=1}^n R_iy_i ) d\text{vol} \\
			&=& n! fd\text{vol}>0.
\end{IEEEeqnarray*}
Here $~~~\widehat{}~~~~$ means the corresponding term is missing, and $d\text{vol}=dx_1 \wedge dy_1 \wedge\cdots\wedge dx_n \wedge dy_n \wedge dt$.

Finally the assertion that $\F$ is the characteristic foliation with respect to $\xi=\ker\alpha$ follows from the construction.
\end{proof}

It turns out that the contact structure $\xi$ constructed in \autoref{lem:exist_nonsing} is, in fact, uniquely determined by $\F$ up to contactomorphism. This is the content of the following lemma and its proof mimics the proof of Weinstein neighborhood theorem.

\begin{lemma}[Uniqueness] \label{lem:unique_nonsing}
Let $(Y,\F) \subset M$ be a foliated submanifold. Suppose $\xi_0,\xi_1$ are contact structures on $M$ such that $\F$ is a Legendrian foliation with respect to both $\xi_0$ and $\xi_1$. Then there exists neighborhoods $\U_0$ and $\U_1$ of $Y$, and a diffeomorphism $\phi:\U_0 \to \U_1$ such that (1) $\phi|_Y=\id_Y$; (2) $\phi_\ast(\xi_0)=\xi_1$.
\end{lemma}

We first recall the Whitney extension theorem as follows.

\begin{theorem}[Whitney Extension Theorem] \label{thm:WET}
Let $Y \subset M$ be a submanifold. Suppose there is a pointwise linear isomorphism $L_p: T_p M \to T_p M$, depending smoothly on $p \in Y$, such that $L_p|_{T_p Y}=\id_{T_p Y}$. Then there exists an embedding $h:\U \to M$ of some neighborhood $\U$ of $Y$ in $M$ such that $h|_Y=\id_Y$ and $dh_p=L_p$ for all $p \in Y$.
\end{theorem}

\begin{proof}[Proof of \autoref{lem:unique_nonsing}]
Choose a contact form $\alpha_i$ for $\xi_i$, $i=0,1$, and a Riemannian metric $g$ on $M$. Let $X$ be a nonvanishing vector field on $Y$ transverse to $\F$. Rescaling $\alpha_i$ if necessary, we can assume $\alpha_0(X)=\alpha_1(X)$.

Now we define a pointwise linear isomorphism $L_p: T_p M \to T_p M$ for each $p \in Y$ in the following two steps.

\vspace{2mm}
\begin{enumerate}
\item[Step 1:] Let $L_p|_{T_p Y}=\id_{T_p Y}$. \vspace{3mm}
\item[Step 2:] Consider the orthogonal decomposition $$\xi_{i,p}=T_p\F \oplus (T_p\F)^{\bot_i},~i=0,1,$$ with respect to $g$, and notice that $T_p\F$ is a Lagrangian subspace in both $\xi_0$ and $\xi_1$. Choose a basis on $T_p\F=\vspan\{e_1,\cdots,e_n\}$, then a simple symplectic linear algebra implies that we can canonically complete the $e_i$'s to a symplectic basis
    	\begin{align*}
		\xi_{0,p}=\vspan\{e_1,\cdots,e_n,f_1,\cdots,f_n\}, \\
		\xi_{1,p}=\vspan\{e_1,\cdots,e_n,g_1,\cdots,g_n\},
	\end{align*}
in the sense that $d\alpha_0(e_i,f_j)=d\alpha_1(e_i,g_j)=\delta_{ij}$ and all the other pairings vanish, where $(T_p\F)^{\bot_0}=\vspan\{f_1,\cdots,f_n\}$ and $(T_p\F)^{\bot_1}=\vspan\{g_1,\cdots,g_n\}$. Now let $L_p(f_i)=g_i$ for $1 \leq i \leq n$.\label{item:J_map}
\end{enumerate}

It is easy to check that the definition of $L_p$ is independent of the choice of a basis on $T_p\F$, and $L_p$ is smoothly varying with $p$ because the symplectic basis completion is canonical. Therefore \autoref{thm:WET} produces a diffeomorphism $\psi:\U_0 \to \U_1$ between neighborhoods of $Y$ which is fixed on $Y$ such that $\psi^\ast(\alpha_1)=\alpha_0$ and $\psi^\ast(d\alpha_1)=d\alpha_0$ on $Y$.

Finally a standard Moser's technique in a (possibly smaller) neighborhood of $Y$ isotopes $\psi$ to $\phi: \U_0 \to \U_1$ which satisfies all the desired properties.
\end{proof}

\subsection{The case of singular Legendrian foliations}\label{sec:nbhd_sing}

Due to our lack of complete knowledge on the singular loci of Legendrian foliations, we will assume for the rest of this section that
\begin{equation}
S(Y)=S_1 \cup \cdots \cup S_k \subset Y
\end{equation}
is a finite disjoint union of $(n-1)$-dimensional closed co-oriented submanifolds. Recall that a codimension 1 singular foliation $\F=\ker \beta$ in $Y$ has {\em normally controlled singularities} if the singular locus $S(\F)$ is a finite union of $(n-1)$-dimensional closed co-oriented submanifolds, and a tubular neighborhood of each path-connected component of $S(\F)$ has the structure of a flat disk bundle (\textit{cf.} Section \ref{subsec:nbhd_sing}) with respect to which the restriction of $\beta$ to the fiber direction is a CCL 1-form in the sense of \autoref{defn:good}.

With the above assumption in mind, we present the existence of a germ of contact structure with given singular foliation as follows.

\begin{lemma}[Existence]\label{lem:exist_sing}
Suppose $(Y,\F)$ is foliated manifold with normally controlled singularities. Then there exists a vector bundle $E \to Y$ and a contact structure $\xi$ on a neighborhood of the 0-section such that $\F$ is the characteristic foliation on $Y$.
\end{lemma}

\begin{proof}
The proof is just assembling several pieces from previous sections together. Namely, we decompose $$Y=N(S_1) \cup \cdots \cup N(S_k) \cup Q$$ where $N(S_i)$ are tubular neighborhoods of each component of $S(\F)$, and $Q$ is the part of $Y$ where $\F$ is nonsingular. We assume that $Q$ intersects each $N(S_i)$ in a collar neighborhood of $\p N(S_i)$.

Over each $N(S_i)$, \autoref{lem:contact_germ} produces a vector bundle $E_i \to N(S_i)$ with fiber $\RR \times T^\ast \h_i$, where $\h_i$ is the horizontal foliation on $N(S_i)$, together with a contact structure $\xi_i$ on the total space such that $\F|_{N(S_i)}$ is the Legendrian foliation. Over $Q$, \autoref{lem:exist_nonsing} produces a vector bundle $E_0 \to Q$ with fiber $T^\ast \F$ and a contact structure $\xi_0$ on the total space such that $\F|_Q$ is the Legendrian foliation.

We only need to observe that the above pieces of bundles and contact structures can be patched together, thanks to \autoref{lem:unique_nonsing} (applied to the overlap between $Q$ and each $N(S_i)$). Finally we note that the gluing map is in general not fiberwise linear due to the application of Moser's technique in \autoref{lem:unique_nonsing}, so to be honest, we will get a fiber bundle with fiber $\RR^n$ after patching. But this is a vector bundle since $\text{Diff}_0(\RR^n)$ deformation retracts onto $\text{GL}(n,\RR)$.
\end{proof}

\begin{remark}
Unlike \autoref{lem:exist_nonsing}, the contact structure constructed in \autoref{lem:exist_sing} is only defined near the 0-section.
\end{remark}

Now we present the uniqueness of the germ of contact structure constructed in \autoref{lem:exist_sing}, which is analogous to \autoref{lem:unique_nonsing}. But due to the presence of singularities, we cannot expect the diffeomorphism to be the identity on $Y$ anymore.

\begin{lemma}[Uniqueness]\label{lem:unique_sing}
Let $(Y,\F) \subset M$ be a singularly foliated submanifold with co-oriented singularities. Suppose $\xi_0,\xi_1$ are contact structures on $M$ such that $\F$ is a Legendrian foliation with respect to both $\xi_0$ and $\xi_1$. Then there exists neighborhoods $\U_0$ and $\U_1$ of $Y$, and a diffeomorphism $\phi:\U_0 \to \U_1$ such that (1) $\phi(Y)=Y$; (2) $\phi_\ast(\xi_0)=\xi_1$.
\end{lemma}

\begin{proof}
Let $\alpha_0,\alpha_1$ be contact forms for $\xi_0,\xi_1$ respectively, and $g$ be a Riemannian metric on $M$. For the sake of simplicity, we assume the singular locus $S$ of $\F$ is connected, and write $Y=N_\epsilon(S)\cup Q$ be the decomposition as before, such that $Q$ intersects $N_\epsilon(S)$ in an arbitrarily small neighborhood of $\p N_\epsilon(S)$. Here $N_\epsilon(S)$ is a tubular neighborhood of $S$ of radius $\epsilon>0$, a small constant which will be determined later. Let $\lambda_i=\alpha_i|_Y, i=0,1,$ be the restricted contact form. Then $\lambda_1=\mu\lambda_0$ away from the singular locus $S(Y)$, where $\mu:Y\setminus S(Y) \to \RR_+$, as they define the same characteristic foliation $\F$. It follows from the proof of \autoref{lem:flat_form} that on $N(S)$, $\h=\ker(d\lambda_0)=\ker(d\lambda_1)$ defines an isotropic codimension 2 foliation on $Y$ whose leaves are contained in the leaves of $\F$. Here we have used the fact that $d\lambda_1=d(\mu\lambda_0)=d\mu\wedge\lambda_0+\mu d\lambda_0$ and $\lambda_0$ vanishes on $\ker(d\lambda_0)$.

Let $X_i$ be the unit normal vector field to $\xi_i$ for $i=0,1$. We define a pointwise linear isomorphism $L_p:T_p M \to T_p M$ for each $p \in Y$ in three steps as follows.

\vspace{3mm}
\begin{enumerate}
\item[Step 1:] Let $L_p|_{T_p Y}=\id_{T_p Y}$. \vspace{3mm}
\item[Step 2:] We define $L_p$ for $p \in Q$. Orthogonally decompose $$\xi_{i,p}=T_p\F \oplus (T_p\F)^{\bot_i},~i=0,1,$$ with respect to the metric $g$ as before. The same argument used in \autoref{lem:unique_nonsing} defines a map $L_p(f_i)=g_i$ for $1 \leq i \leq n$, where
	\begin{align*}
		\xi_{0,p}=\vspan\{e_1,\cdots,e_n,f_1,\cdots,f_n\}, \\
		\xi_{1,p}=\vspan\{e_1,\cdots,e_n,g_1,\cdots,g_n\},
	\end{align*}
are the canonical symplectic basis adapted to the orthogonal decompositions.
\item[Step 3:] Now we define $L_p$ for $p \in N_\epsilon(S)$. Given $p \in N_\epsilon(S)\setminus S$, pick a basis on $T_p\h=\vspan\{e_1,\cdots,e_{n-1}\}$, and let $e_n$ generate $T_p\F/T_p\h$ $\simeq \RR$. Let $(T_p\F)^{\bot_0}=\vspan\{f_1,\cdots,f_n\}$ be the canonical dual basis with respect to $d\alpha_0$, i.e., $d\alpha_0(e_i,f_j)=\delta_{ij}$ and all the other pairings vanish. Similarly let $(T_p\F)^{\bot_1}=\vspan\{g_1,\cdots,g_n\}$ with respect to $d\alpha_1$. Let $\chi(r):[0,\epsilon] \to [0,\epsilon]$ be a smooth increasing function such that $\chi(r)=0$ for $r$ close to $0$ and $\chi(r)=\epsilon$ for $r$ close to $\epsilon$. Denote $r_p$ the distance from $p$ to $S$. Finally we define $L_p(f_i)=g_i$ for $1 \leq i \leq n-1$, and
        \begin{equation} \label{eqn:normal_map_sing}
        L_p((\epsilon-\chi(r_p))X_0+\chi(r_p)f_n)=(\epsilon-\chi(r_p))X_1+\chi(r_p)g_n,
        \end{equation}
    for any $p \in N_\epsilon(S)\setminus S$. Our definition obviously extends to $S$ by asking $L_p(X_0)=X_1$ for $p \in S$.
\end{enumerate}

One can check that the definition of $L_p$ is independent of the choice of various basis, well-defined on the overlaps, and depends smoothly on $p \in Y$. In particular we note that for small $\epsilon$, the vector $(\epsilon-\chi(r_p))X_0+\chi(r_p)f_n$, together with $\{f_1,\cdots,f_{n-1}\}$, span $T_p M/T_p Y$. Similarly for $(\epsilon-\chi(r_p))X_1+\chi(r_p)g_n$.

Applying \autoref{thm:WET}, we get a diffeomorphism $\psi:\U_0 \to \U_1$ between neighborhoods of $Y$ which is fixed on $Y$, such that $d\psi|_Y=L$. Abusing notations, we still use $\alpha_1$ to denote $\psi^\ast(\alpha_1)$. Observe that $\alpha_0=\alpha_1$ and $d\alpha_0=d\alpha_1$ on $Q$ by construction, but they do not necessarily agree on $N_\epsilon(S)$ (in fact not even on $S$ because $d\lambda_0 \neq d\lambda_1$ in general).

\begin{claim}\label{claim:path_of_contact}
The form $\alpha_t=(1-t)\alpha_0+t\alpha_1$ is contact on $Y$, and therefore on a neighborhood of $Y$, for all $0 \leq t \leq 1$.
\end{claim}

\begin{proof}[Proof of \autoref{claim:path_of_contact}]
It is easy to see that $\alpha_t$ is contact along $S$ since both $d\lambda_0$ and $d\lambda_1$ define positive area forms\footnote{Here the matching of the co-orientation of $S\subset Y$ is crucial.} in the normal plane to $S$. Therefore we only need to verify the contact condition on $N_\epsilon(S) \setminus S$. Let $\n$ be the unit normal vector field to $\F$ in $N_\epsilon(S)\setminus S$. To understand the pull-back contact structure $\alpha_1$ on $Y$, we need to solve for $L_p^{-1}(g_n)$ using (\ref{eqn:cov_const_form}). Observe that
\begin{equation}\label{eqn:x_0,x_1}
X_0=C_1(\n-\sum_{i=1}^n \langle\n,f_i\rangle f_i), \text{~~and~~} X_1=C_2(\n-\sum_{i=1}^n \langle\n,g_i\rangle g_i),
\end{equation}
where $C_1=1/||\n-\sum_{i=1}^n \langle\n,f_i\rangle f_i||$, $C_2=1/||\n-\sum_{i=1}^n \langle\n,g_i\rangle g_i||$, and $\langle \cdot,\cdot \rangle$ denotes the inner product induced by $g$. Plug (\ref{eqn:x_0,x_1}) into (\ref{eqn:normal_map_sing}) and use the fact that $L_p(\n)=\n$, we obtain for each $p \in N_\epsilon(S)\setminus S$
\begin{equation*}
B'L^{-1}_p(g_n)=\sum_{i=1}^{n-1}A'_i f_i + B''f_n+C'\n
\end{equation*}
where $A'_i=(1-r_p)\sum_{i=1}^{n-1}\langle\n,C_2 g_i-C_1 f_i\rangle$, $B'=C_2(\chi(r_p)-\langle\n,g_n\rangle)$, $B''=C_1(r_p-\langle\n,f_n\rangle)$, and $C'=(C_1-C_2)(1-r_p)$. Now we choose small $\epsilon>0$ such that $\langle\n,f_n\rangle>2\epsilon$ and $\langle\n,g_n\rangle>2\epsilon$ for every $p \in N_\epsilon(S)$. This is possible because both $d\lambda_0$ and $d\lambda_1$ restricts to positive area forms on the normal plane to $S$.

By letting $A_i=A'_i/B'$, $B=B''/B'$ and $C=C'/B'$, we get
\begin{equation}\label{eqn:solve_gn}
L^{-1}_p(g_n)=\sum_{i=1}^{n-1}A_i f_i + Bf_n+C\n
\end{equation}
such that $B>0$. Therefore we have computed the symplectic basis $$\xi_1=\vspan\{e_1,\cdots,e_n,f_1\cdots,f_{n-1},L^{-1}_p(g_n)\}$$ with respect to the pull-back symplectic form $d\alpha_1$. Now a straightforward calculation using (\ref{eqn:solve_gn}) verifies the assertion of \autoref{claim:path_of_contact}, and we leave the details to the interested reader.
\end{proof}

Finally the standard Moser's technique, applied to the path $\alpha_t$, isotopes $\psi$ to a diffeomorphism $\phi:\U_0 \to \U_1$ which satisfies all the desired properties.
\end{proof}

Finally the conclusion of \autoref{thm:nbhd} is the combination of the conclusions of \autoref{lem:exist_nonsing}, \autoref{lem:unique_nonsing}, \autoref{lem:exist_sing} and \autoref{lem:unique_sing}.

\bibliography{mybib}
\bibliographystyle{amsalpha}

\end{document}